\documentclass[10pt,a4paper]{article}

\usepackage{amsthm}
\usepackage{amsmath}

\usepackage{amsfonts}
\usepackage{amssymb}
\usepackage{amscd}
\usepackage{graphicx}
\usepackage{color}
\usepackage{url,paralist}
\usepackage[margin=30mm]{geometry}

\newcommand\Theorem{Theorem}
\newcommand\Figure{Figure}
\newcommand\Section{Section}
\newcommand\Sections{Sections}
\newcommand\Corollary{Corollary}

\newtheorem{theorem}{Theorem}[section]
\newtheorem{proposition}[theorem]{Proposition}
\newtheorem{lemma}[theorem]{Lemma}
\newtheorem{corollary}[theorem]{Corollary}

\newcommand{\barany}{B\'{a}r\'{a}ny}
\newcommand{\blagojevic}{Blagojevi\'{c}}
\newcommand{\cech}{\v{C}ech}
\newcommand{\matousek}{Matou\v{s}ek}

\newcommand{\szucs}{Sz{\H{u}}cs}
\newcommand{\vrecica}{Vre\'{c}ica}

\newcommand*{\longhookrightarrow}{\ensuremath{\lhook\joinrel\relbar\joinrel\rightarrow}}
\newcommand*{\longlongrightarrow}{\ensuremath{\relbar\joinrel\relbar\joinrel\rightarrow}}

\newcommand{\RR}{\mathbb{R}} 
\newcommand{\ZZ}{\mathbb{Z}} 
\newcommand{\FF}{\mathbb{F}} 
\newcommand{\wo}{\backslash} 
\newcommand{\To}{\longrightarrow} 
\newcommand{\toG}[1]{\longrightarrow_{#1}} 
\newcommand{\tof}[1]{\stackrel{#1}{\longrightarrow}} 
\newcommand{\ltof}[1]{\stackrel{#1}{\longlongrightarrow}} 
\newcommand{\iso}{\cong} 
\newcommand{\bd}{\partial} 
\newcommand{\ind}{\textnormal{Ind}} 
\newcommand{\pt}{\textnormal{pt}} 
\newcommand{\incl}{\hookrightarrow} 
\newcommand{\lincl}{\longhookrightarrow}
\newcommand{\st}{~:~} 
\newcommand{\simplex}{\Delta} 
\newcommand{\vertices}{\textnormal{vert}} 
\newcommand{\cohdim}{\textnormal{cohdim}} 
\newcommand{\wt}{\widetilde}

\begin{document}

\title {A tight colored Tverberg theorem for maps to manifolds\setcounter{footnote}{1}%
\footnote{Topology and its Applications 158(12), 2011, 1445-1452.}}

\author{%
\setcounter{footnote}{0}
Pavle V. M. Blagojevi\'{c}\thanks{%
The research leading to these results has received funding from the European Research
Council under the European Union's Seventh Framework Programme (FP7/2007-2013) /
ERC Grant agreement no.~247029-SDModels. Also supported by the grant ON 174008 of the Serbian
Ministry of Science and Environment.} \\
\small Mathemati\v cki Institut\\
 SANU\\
\small Knez Michailova 36\\
\small 11001 Beograd, Serbia\\
\small \url{pavleb@mi.sanu.ac.rs} \and \setcounter{footnote}{0}
Benjamin Matschke$^{*}$%
\thanks{$^{*}$Supported by Deutsche Telekom Stiftung.}\\
\small Institute of Mathematics\\
\small FU Berlin\\
\small Arnimallee 2\\
\small 14195 Berlin, Germany\\
\small \url{matschke@math.fu-berlin.de} \and \setcounter{footnote}{0}
G\"unter M. Ziegler$^{**}$%
\thanks{$^{**}$%
The research leading to these results has received funding from the European Research
Council under the European Union's Seventh Framework Programme (FP7/2007-2013) /
ERC Grant agreement no.~247029-SDModels.} \\
\small Institute of Mathematics\\
\small FU Berlin\\
\small Arnimallee 2\\
\small 14195 Berlin, Germany\\
\small \url{ziegler@math.fu-berlin.de}}
\date{April 15, 2011}

\maketitle

\begin{abstract}\noindent

We prove that any continuous map of an $N$-dimensional simplex $\simplex_N$ with colored vertices to a $d$-dimensional manifold~$M$ must map $r$ points from disjoint rainbow faces of $\simplex_N$ to the same point in~$M$:
For this we have to assume that $N\geq (r-1)(d+1)$, no $r$ vertices of $\simplex_N$ get the same color, and our proof needs that $r$ is a prime.
A face of $\simplex_N$ is a \emph{rainbow face} if all vertices have different colors.

This result is an extension of our recent ``new colored Tverberg theorem'',
the special case of $M=\RR^d$.
It is also a generalization of Volovikov's 1996 topological
Tverberg theorem for maps to manifolds, which arises when all color classes have size~$1$ (i.e., without color constraints);
for this special case Volovikov's proof, as well as ours, work when $r$ is a prime power.
\end{abstract}

\section{Introduction}

Recently, we formulated a new version of the 1992
``colored Tverberg conjecture'' by \barany\ and Larman
\cite{BL92}, and proved this new version in the case of primes.

\begin{theorem}[{Tight colored Tverberg theorem} \cite{BMZ09a}]
\label{thmMainForRd}
For $d\ge 1$ and a prime $r\ge 2$, set $N:=(d+1)(r-1)$,
and let the $N+1$ vertices of an $N$-dimensional simplex $\simplex_N$
be colored such that all color classes are of size at most $r-1$.

Then for every continuous map $f:\simplex_N\rightarrow \RR^d$,
there are $r$ disjoint faces $F_1,\dots,F_r$ of $\simplex_N$
such that the vertices of each face $F_i$ have all different colors,
and such that the images under $f$ have a point in common:
$f(F_1)\cap\dots\cap f(F_r)\neq\emptyset$.
\end{theorem}

Here a \emph{coloring} of the vertices of the simplex $\simplex_N$ is a partition
of the vertex set into color classes, $C_1\uplus\dots\uplus C_m$.
The condition $|C_i|\le r-1$ implies that there are at least $d+2$
different color classes.
In the following, a face all whose vertices have different colors,
$|F_j\cap C_i|\le 1$ for all~$i$, will be called a \emph{rainbow face}.

Theorem~\ref{thmMainForRd} is tight in the sense that it fails
for maps of a simplex of smaller dimension, or if some $r$
vertices of the simplex have the same color.
It implies an optimal result for the
\barany--Larman conjecture in the case where $r+1$
is a prime, and an asymptotically-optimal bound in general; see~\cite[Corollary 2.4 and 2.5]{BMZ09a}.
The special case where all vertices of $\simplex_N$ have different
colors, $|C_i|=1$, is the topological Tverberg theorem of
\barany, Shlosman \& \szucs~\cite{BSS81}.

\smallskip

In this paper we present an extension of Theorem~\ref{thmMainForRd}
that treats continuous maps $\simplex_N\to M$ from the $N$-simplex to an
arbitrary $d$-dimensional manifold $M$ in place of~$\RR^d$.

\begin{theorem}[{Tight colored Tverberg theorem for} $M$]
\label{thmMain}
For $d\ge 1$ and a prime $r\ge 2$, set $N:=(d+1)(r-1)$,
and let the $N+1$ vertices of an $N$-dimensional simplex $\simplex_N$
be colored such that all color classes are of size at most $r-1$.

Then for every continuous map $f:\simplex_N\rightarrow M$
to a $d$-dimensional manifold $M,$ the simplex $\simplex_N$ has $r$ disjoint rainbow faces
whose images under $f$ have a point in common.
\end{theorem}

\smallskip

\Theorem~\ref{thmMain} without color constraints (that is,
when all color classes are of size~$1$, and thus all faces are rainbow faces)
was previously obtained by Volovikov \cite{Vol96}, using different methods.
His proof (as well as ours in the case without color constraints) works
for prime powers~$r$; see \Section~\ref{secUncoloredCase}.

The prime power case for the colored version,
\Theorem~\ref{thmMain}, seems however out of reach at this point, even in
the case $M=\RR^d$.
Similarly, there currently does not seem to be
a viable approach to the case without color constraints,
even for $M=\RR^d$, when $r$ is not a prime power.
This is the remaining open case of the topological Tverberg conjecture \cite{BSS81}.

The conclusion of Theorem~\ref{thmMain} remains
valid if we only consider a continuous map $f: R\rightarrow M$,
where $R=C_1*\ldots*C_m$ denotes the subcomplex of rainbow faces in~$\simplex_N$.
This is non-trivial in general.
See the discussion in \Section~\ref{secProofOfThmWithSmallDomain}.

\section{Proof}

We prove \Theorem~\ref{thmMain} in two steps:
\begin{compactitem}[$\bullet$]
\item First, a geometric reduction lemma implies that it suffices to consider only manifolds $M$ that are of the form $M=\wt M\times I^g$, where $I=[0,1]$ and $\wt M$ is another manifold. More precisely we will need for the second step that
\begin{equation}
\label{eqDimVsCohdim}
(r-1)\dim(M)>r\cdot\cohdim(M),
\end{equation}
where $\cohdim(M)$ is the cohomology dimension of $M$.
This is done in \Section~\ref{secReductionLemma}.
\item In the second step, we can assume \eqref{eqDimVsCohdim} and prove \Theorem~\ref{thmMain} for maps $\simplex_N\to \wt M$ via the configuration space/test map scheme and Fadell--Husseini index theory, see \Sections~\ref{secCS/TM} and~\ref{secIndexOfM^}.
\end{compactitem}
In the second step we rely on the computation of the Fadell--Husseini index of joins of chessboard complexes that we obtained in \cite{BMZ09b}.

\subsection{A geometric reduction lemma} \label{secReductionLemma}

In the proof of Theorem~\ref{thmMain} may assume that $M$ satisfy the above inequality \eqref{eqDimVsCohdim} by using the following reduction lemma repeatedly.

\begin{lemma}
\Theorem~\ref{thmMain} for parameters $(d,r,M,f)$ can be derived from the case with parameters $(d',r',M',f')=(d+1,r,M\times I,f')$, where the continuous map $f'$ is defined in the following proof.
\end{lemma}

\begin{proof}
Suppose we have to prove the theorem for the parameters $(d,r,M,f)$.
Let $d'=d+1$, $r'=r$, and $M'=M\times I$.
Then $N':=(d'+1)(r-1)=N+r-1$.
Let $v_0,\ldots,v_N,v_{N+1},\ldots,v_{N'}$ denote the vertices of $\simplex_{N'}$.
We regard $\simplex_N$ as the front face of $\simplex_{N'}$ with vertices $v_0,\ldots,v_{N}$.
We give the new vertices $v_{N+1},\ldots,v_{N'}$ a new color.
Using barycentric coordinates, define a new map $f':\simplex_{N'}\to M'$ by
\[
\lambda_0 v_0+\ldots+\lambda_{N'}v_{N'}
\longmapsto
\left(f(\lambda_0v_0+\ldots+\lambda_{N-1}v_{N-1}+(\lambda_N+\ldots+\lambda_{N'})v_n),\lambda_{N+1}+\ldots+\lambda_{N'}\right).
\]

\noindent Suppose we can show \Theorem~\ref{thmMain} for the parameters $(d',r',M',f')$.
That is, we found a Tverberg partition $F'_1,\ldots,F'_r$ for these parameters.
Put $F_i:=F'_i\cap \simplex_N$.
Since $f'$ maps the front face $\simplex_N$ to $M\times\{0\}$ and since $\simplex_{N'}$ has only $r-1<r$ vertices more than $\simplex_N$, already the $F_i$ will intersect in $M\times\{0\}$.
Hence the $r$ faces $F_1,\ldots,F_r$ form a solution for the original parameters $(d,r,M,f)$.
This reduction is sketched in \Figure~\ref{figReductionLemma}.
\begin{figure}[tbh]
\centering
\input{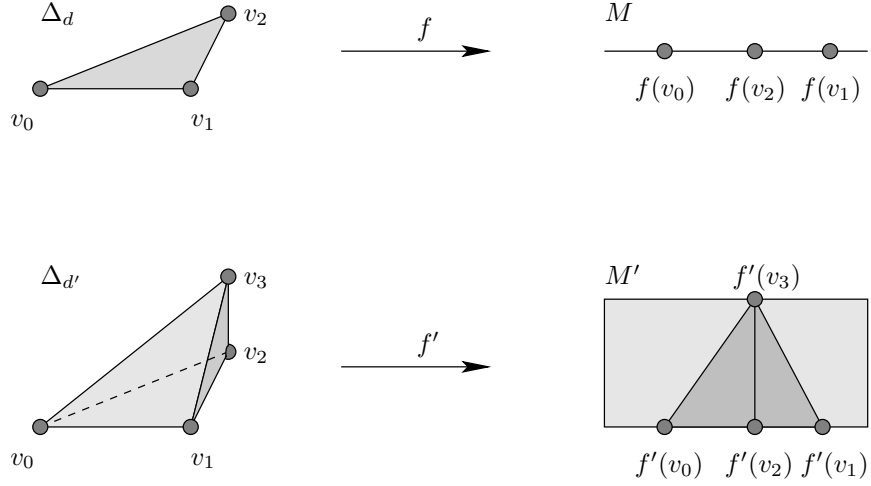}
\caption{Exemplary reduction in the case $d=1$, $r=2$, $N=2$.}
\label{figReductionLemma}
\end{figure}
\end{proof}

\noindent If the reduction lemma is applied $g=1+\big\lfloor{\frac{d}{r-1}}\big\rfloor$ times, the problem is reduced from the arbitrary parameters $(d,r,M,f)$ to parameters $(d'',r'',M'',f'')$ where $M''=M\times I^g$.
Thus $M''$ has vanishing cohomology in its $g$ top dimensions.
Therefore $(r-1)\dim(M'')>r\cdot\cohdim(M'')$.

\smallskip
Having this reduction in mind, in what follows we may simply assume that the manifold $M$ already satisfies inequality \eqref{eqDimVsCohdim}.

\subsection{The configuration space/test map scheme} \label{secCS/TM}

Suppose we are given a continuous map
\[
f:\simplex_N\To M,
\]
and a coloring of the vertex set $\vertices(\simplex_N)=[N+1]=C_0\uplus\dots\uplus C_m$
such that the color classes $C_i$ are of size $|C_i|\le r-1$.
We want to find a colored Tverberg partition, that is, pairwise disjoint rainbow faces $F_1,\dots,F_r$ of $\simplex_N$, $|F_j\cap C_i|\leq 1$, whose images under $f$ intersect.

The test map $F$ is constructed using~$f$ in the following way.
Let $f^{*r}:(\simplex_N)^{*r}\toG{\ZZ_r} M^{*r}$ be the $r$-fold join of~$f$.
Since we are interested in pairwise disjoint faces $F_1,\dots,F_r$, we restrict the domain of $f^{*r}$ to the $r$-fold $2$-wise deleted join of $\simplex_N$,
$(\simplex_N)^{*r}_{\Delta (2)} = [r]^{*(N+1)}$.
This is the subcomplex of $(\simplex_N)^{*r}$ consisting of all joins $F_1*\ldots*F_r$ of pairwise disjoint faces.
(See \cite[Chapter 5.5]{Mat07} for an introduction to these notions.)
Since we are interested in colored faces $F_j$, we restrict the domain further to the subcomplex
\[
R^{*r}_{\Delta(2)}=\left(C_0 *\dots * C_m\right)^{*r}_{\Delta(2)}=[r]^{*|C_0|}_{\Delta(2)}*\dots*[r]^{*|C_m|}_{\Delta(2)}.
\]
This is the subcomplex of $(\simplex_N)^{*r}$ consisting of all joins $F_1*\ldots*F_r$ of pairwise disjoint rainbow faces.
The space $[r]^{*k}_{\Delta(2)}$ is known as the \emph{chessboard complex} $\simplex_{r,k}$, \cite[p. 163]{Mat07}.
We write
\begin{equation}
\label{eqDefOfTestSpaceK}
K:=(\simplex_{r,|C_0|}) * \dots * (\simplex_{r,|C_m|}).
\end{equation}
Hence we get a \emph{test map}
\[
F': K \toG{\ZZ_r} M^{*r}.
\]
Let $T_{M^{*r}}:=\{\sum_{i=1}^{r} \frac{1}{r}\cdot x ~:~ x\in M\}$ be the thin diagonal of $M^{*r}$.
Its complement $M^{*r}\wo T_{M^{*r}}$ is called the topological $r$-fold $r$-wise deleted join of $M$ and it is denoted by $M^{*r}_{\Delta(r)}$.

The preimages $(F')^{-1}(T_{M^{*r}})$ of the thin diagonal correspond exactly to the colored Tverberg partitions.
Hence the image of $F'$ intersects the diagonal if and only if $f$ admits a colored Tverberg partition.

Suppose that $f$ admits no colored Tverberg partition, then the test map $F'$ induces a $\ZZ_r$-equivariant map
\begin{equation}
\label{eqTestmapF}
F: K \toG{\ZZ_r} M^{*r}_{\Delta(r)}.
\end{equation}
We will derive a contradiction to the existence of such an equivariant map using the Fadell--Husseini index theory.

\subsection{The Fadell--Husseini index}

In this section we review equivariant cohomology of $G$-spaces via the Borel construction. We refer the reader to \cite[Chap. V]{AM94} and \cite[Chap. III]{Die87} for more details.

Let in the following $H^*$ denote singular or \cech\ cohomology with $\FF_r$-coefficients, where $r$ is prime, and $G$ a finite group.
Let $EG$ be a contractible free $G$-CW complex, for example the infinite join $G*G*\ldots$, suitably topologized. The quotient $BG:=EG/G$ is called the \emph{classifying space of $G$}.
To every $G$-space $X$ we can associate the \emph{Borel construction} $EG\times_G X:=(EG\times X)/G$, which is the total space of the fibration $X\incl EG\times_G X\tof{pr_1} BG$.

\smallskip

The \textit{equivariant cohomology} of a $G$-space $X$ is defined as the ordinary cohomology of the Borel construction,
\[
H^*_G(X)\ :=\ H^*(EG\times_G X).
\]

\smallskip

If $X$ is a $G$-space, we define the \emph{cohomological index} of $X$, also called the
\emph{Fadell--Husseini index} \cite{FH87,FH88}, to be the kernel of the map in cohomology induced by the projection from $X$ to a point,
\[
\ind_G(X)\ :=\ \ker\big(H^*_G(\pt)\tof{p^*} H^*_G(X)\big)
           \ \subseteq\ H^*_G(\pt).
\]
The cohomological index is monotone in the sense that if there is a $G$-map $X\toG{G}Y$ then
\begin{equation}
\label{eqFHMonotonicity}
\ind_G(X)\supseteq\ind_G(Y).
\end{equation}

\smallskip

If $r$ is odd then the cohomology of~$\ZZ_r$ with $\FF_r$-coefficients as an $\FF_r$-algebra is
\[
H^*(\ZZ_r)=H^*(B\ZZ_r)\iso \FF_r[x,y]/ {(y^2)},
\]
where $\deg(x)=2$ and $\deg(y)=1$.
If $r$ is even, then $r=2$ and $H^*(\ZZ_r)\iso \FF_2[t]$, $\deg t=1$.

\smallskip

The index of the complex $K$ was computed in \cite[Corollary 2.6]{BMZ09b}:
\begin{theorem}
\label{thmIndexOfK}
$\ind_{\ZZ_r}(K) = H^{*\ge N+1}(B\ZZ_r)$.
\end{theorem}
Therefore in the proof of \Theorem~\ref{thmMain} it remains to show that $\ind_{\ZZ_r}(M^{*r}_{\Delta(r)})$ contains a non-zero element in dimension less or equal to $N$.
Indeed, the monotonicity of the index \eqref{eqFHMonotonicity} implies the non-existence of a test map \eqref{eqTestmapF}, which in turn implies the existence of a colored Tverberg partition.

Let us remark that the index of $K$ becomes larger as an ideal than in Theorem~\ref{thmIndexOfK} if just one color class $C_i$ has more than $r-1$ elements.
That is, in this case our proof of Theorem~\ref{thmMain} does not work anymore.
In fact, for any $r$ and $d$ there exist $N+1$ colored points in $\RR^d$ such that one color class is of size $r$ and all other color classes are singletons that admit no colored Tverberg partition.

\subsection{The index of the deleted join of the manifold} \label{secIndexOfM^}

We have inclusions
\[
T_{M^{*r}} \ \lincl\ \left\{\sum\lambda_i x\in M^{*r} \st \lambda_i>0,\sum\lambda_i=1,x\in M\right\}\iso M\times\simplex_{r-1}^\circ \ \lincl\ M^{*r},
\]
where $\Delta_{r-1}^\circ$ denotes the open $(r-1)$-simplex.
Since $M$ is a smooth $\ZZ_r$-invariant manifold, $T_{M^{*r}}$ has a $\ZZ_r$-equivariant tubular neighborhood in $M^{*r}$; see \cite{Bre72}.
Its closure can be described as the disk bundle $D(\xi)$ of an equivariant vector bundle $\xi$ over $M$.
We denote its sphere bundle by $S(\xi)$.
The fiber $F$ of $\xi$ is as a $\ZZ_r$-representation the $(d+1)$-fold sum of $W_r$, where $W_r=\{x\in\RR[\ZZ_r]~:~ x_1+\ldots+x_r=0\}$ is the augmentation ideal of $\RR[\ZZ_r]$.

\smallskip

The representation sphere $S(F)$ is of dimension $N-1$.
It is a free $\ZZ_r$-space, hence its index is
\begin{equation}\label{eq:SphereIndex}
\ind_{\ZZ_r}(S(F)) = H^{*\ge N}(B\ZZ_r).
\end{equation}
This can be directly deduced from the Leray--Serre spectral sequence associated to the Borel construction $E\ZZ_r\times_{\ZZ_r} S(F)\to B\ZZ_r$, noting that the images of the differentials to the bottom row give precisely the index of $S(F)$, which can be seen from the edge-homomorphism.
For background on Leray--Serre spectral sequences we refer to \cite[Chapter 5, 6]{McC01}.

The Leray--Serre spectral sequence associated to the fibration $S(\xi)\to M$ collapses at $E_2$, since $N=(r-1)(d+1)\geq d+1$ and hence there is no differential between non-zero entries.
Thus the map $i^*:H^{N-1}(S(\xi))\to H^{N-1}(S(F))$ induced by inclusion is surjective.

\smallskip

The Mayer--Vietoris sequence associated to the triple $(D(\xi),M^{*r}_{\Delta(r)},M^{*r})$ contains the subsequence
\[
H^{N-1}(M^{*r}_{\Delta(r)})\oplus H^{N-1}(D(\xi))\ltof{j^*+k^*} H^{N-1}(S(\xi))\ltof{\delta} H^N(M^{*r}).
\]
We see that $H^N(M^{*r})$ is zero: This follows from the formula
\[
\widetilde{H}^{*+(r-1)}(M^{*r})\iso \left(\widetilde{H}^*(M)\right)^{\otimes r},
\]
as long as $N-(r-1)>re$, where $e$ is the cohomological dimension of $M$.
This inequality is equivalent to $(r-1)d>re$, which is our assumption \eqref{eqDimVsCohdim}.
Hence we can assume that $H^N(M^{*r})=0$.

\smallskip

Furthermore inequality \eqref{eqDimVsCohdim} implies that $N-1\geq d>\cohdim(M)$.
Hence the term $H^{N-1}(D(\xi))=H^{N-1}(M)$ of the sequence is zero as well.

Thus the map $j^*:H^{N-1}(M^{*r}_{\Delta(r)})\to H^{N-1}(S(\xi))$ is surjective.
Therefore the composition $(j\circ i)^*:H^{N-1}(M^{*r}_{\Delta(r)})\to H^{N-1}(S(F))$ is surjective as well.
We apply the Borel construction functor $E\ZZ_r\times_{\ZZ_r}(\_)\to B\ZZ_r$ to this map and apply Leray--Serre spectral sequences; see \Figure~\ref{figSSofSFandDJM}.

\smallskip

\begin{figure}[tbh]
\centering
\input{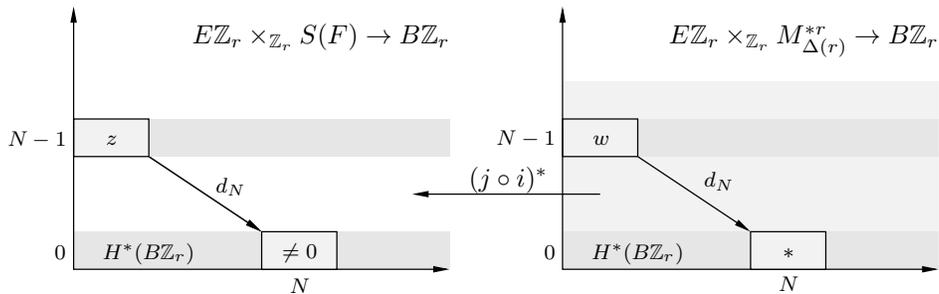}
\caption{We associate to the map $S(F)\tof{j\circ i}M^{*r}_{\Delta(r)}$ the Borel constructions and spectral sequences to deduce that $M^{*r}_{\Delta(r)}$ contains a non-zero element in dimension $N$.}
\label{figSSofSFandDJM}
\end{figure}

At the $E_2$-pages, the generator $z$ of $H^{N-1}(S(F))$ has a preimage $w$ since $(j\circ i)^*$ is surjective.
At the $E_N$-pages $(j\circ i)^*(d_N(w))=d_N(z)$, which is non-zero by \eqref{eq:SphereIndex}.
Hence $d_N(w)\neq 0$, which is an element in the kernel of the edge-homomorphism $H^*(B\ZZ_r)\to H^*_{\ZZ_r}(M^*_{\Delta(r)})$.

Therefore, the index of $M^{*r}_{\Delta(r)}$ contains a non-zero element in dimension $N$.
This completes the proof of \Theorem~\ref{thmMain}. \qed

\section{Remarks}

\subsection{Theorem \ref{thmMain} strictly generalizes Theorem \ref{thmMainForRd}} \label{secNonFactorizability}

One may ask whether Theorem~\ref{thmMain} can be reduced to Theorem~\ref{thmMainForRd} by factorizing the given map $f:\simplex_N\to M$ over $\RR^d$,
\[
f:\simplex_N\tof{f'}\RR^d\to M.
\]
In this case Theorem~\ref{thmMainForRd} immediately implies Theorem~\ref{thmMain}.
However this is not always possible.
\begin{proposition}
Let $f$ be the composed map $\simplex_3\to S^3\to S^2$ that first quotients out the boundary of $\simplex_3$ and then sends $S^3$ to $S^2$ via the Hopf map.
Then $f$ does not factor over $\RR^d$.
\end{proposition}

The following proof is due to Elmar Vogt.

\begin{proof}
Suppose that $f$ factors as $f:\simplex_N\tof{f'}\RR^d\tof{g} M$.
Let $h:S^3\to S^2$ denote the Hopf map.
Let $z:=h^{-1}(n)\subset S^3$, where $n$ is the north pole of $S^2$.
We think of $z$ being the closure of the $z$-axis in the stereographic projection of $S^3$ to $\RR^3$ in the one-point compactification of $\RR^3$.
Let $D$ be the halfspace $\{(x,y,z)\st x>0,\ y=0\}$, which is a disc in $S^3$ whose boundary is $z$.
Then all fibers $h^{-1}(x)$ other than $z$ intersect $D$ transversally.
In particular, $h$ maps $D$ homeomorphically to $S^2\wo\{n\}$.
Then $f'$ also maps $D$ (regarded as a disc in $\simplex_3$) homeomorphically to a set $D'\subset \RR^2$.
Moreover, for every $x\neq n$, $f'(h^{-1}(x))$ is a singleton in $D'$.
Further, $D'$ is bounded, since $\simplex_3$ is compact.
Let $p\in z$ and let $(p_i)$ be a sequence in $D$ converging to $p$.
Then the fibers $h^{-1}(h(p_i))$ contain sequences of points that come close to any other point of $z$.
By continuity, $f'(z)$ must be a singleton in $\RR^2$ as well.
But $f'(z)$ must also be the boundary of $D'$ which is bounded and homeomorphic to an open disc, which gives a contradiction.
\end{proof}

\subsection{The case without color constraints} \label{secUncoloredCase}

Suppose we color the vertices of $\simplex_N$ in \Theorem~\ref{thmMain} with pairwise distinct colors.
Then all faces of $\simplex_N$ are rainbow faces, hence the condition of being a rainbow face is empty.
This case was already treated by Volovikov, in a slightly stronger version.

\begin{theorem}[Volovikov \cite{Vol96}]
\label{thmVolovokovsTopTverberg}
Let $d\geq 1$, let $r=p^k$ be a prime power, $N:=(d+1)(r-1)$, and $f:\bd\simplex_N\to M$ be a continuous map from the boundary $N$-simplex to a $d$-dimensional topological manifold.
If $r=2$ then we further assume that the degree of $f$ is even.
Then $\simplex_N$ has $r$ disjoint rainbow faces whose images under $f$ intersect.
\end{theorem}

Our proof given in this paper works also for prime powers $r=p^k$ in the
case without color constraints,
since then
\begin{compactitem}[$\bullet$]
\item the configuration space is the join
$[r]^{*(N+1)}$, which is $(N-1)$-connected and $(\ZZ_p)^k$-free,
hence its index is $H^{*\ge N+1}(B((\ZZ_p)^k))$, and
\item the group  $(\ZZ_p)^k$ acts fixed point freely on the sphere $S(F)$ and
$\ind_{(\ZZ_p)^k}(S(F))$ consequently contains an element of degree $N$, particularly
$d_N(z)$ in the notation of \Section~\ref{secIndexOfM^}.
\end{compactitem}

\subsection{Reduction to the subcomplex of rainbow faces} \label{secProofOfThmWithSmallDomain}

One could ask whether $\simplex_N$ in \Theorem~\ref{thmMain} can be replaced by
the subcomplex $R$ that consists of all rainbow faces.
The methods of this paper seem to establish this only if we assume that sufficiently many colors are used.
(The assumptions of \Theorem~\ref{thmMain} imply that the $N+1$ vertices of $\simplex_N$ are colored with at least $\big\lceil{\frac{N+1}{r-1}}\big\rceil=d+2$ colors.)

\begin{corollary}
\label{theoremMainWithSmallDomain2}%
Let $d\geq 1$, $r\geq 2$ prime, and $N:=(d+1)(r-1)$.
Let the vertices of $\simplex_N$ be colored with at least $d+3+\big\lfloor{\frac{d}{r-1}}\big\rfloor=d+2+g$ colors such that all color classes $C_i$ are of size $|C_i|\leq r-1$.
Let $R$ be the subcomplex of $\simplex_N$ consisting of all rainbow faces.
Let $f:R\to M$ be a continuous map from $R$ to a $d$-dimensional manifold $M$.
Then $R$ has $r$ disjoint faces whose images under $f$ intersect.
\end{corollary}

The proof of \Corollary~\ref{theoremMainWithSmallDomain2} is analogous to that of \Theorem~\ref{thmMain}.
The main change occurs in the reduction to the case where the manifold is $M'=M\times I^g$, see \Section~\ref{secReductionLemma}.
Here one needs to be a bit more careful, because the reduction might not be possible.
Instead of letting $f$ send the $r-1$ new vertices of $\simplex_{N'}$ to points above $f(v_N)$ and giving them a new color, we send them above the images of possibly different vertices and color them with the same color as the vertex below.
This has to be done in such a way that all new color classes are still of size less than $r$.
This is possible since the number of used colors is at least $d+2+g$.

\medskip

Using more advanced machinery we can even prove \Theorem~\ref{thmMain} even for maps $f:R\to M$, without the constraint on the number of color classes from Corollary~\ref{theoremMainWithSmallDomain2}.
\begin{corollary}
\label{theoremMainWithSmallDomainFullResult}%
Let $d\geq 1$, $r\geq 2$ prime, and $N:=(d+1)(r-1)$.
Let the vertices of $\simplex_N$ be colored such that all color classes $C_i$ are of size $|C_i|\leq r-1$.
Let $R$ be the subcomplex of $\simplex_N$ consisting of all rainbow faces.
Let $f:R\to M$ be a continuous map from $R$ to a $d$-dimensional manifold $M$.
Then $R$ has $r$ disjoint faces whose images under $f$ intersect.
\end{corollary}

To prove this, we use the deleted-product scheme.
The deleted-join scheme amounts to show that the test map~\eqref{eqTestmapF} does not exist.
In the corresponding deleted-product scheme one instead has to show that the test map
\begin{equation}
\label{eqTestmapFProduct}
f^\times: R^r_{\Delta(2)}\to M^r_{\Delta(r)}
\end{equation}
does not exist.
For this to do, we calculate the index of $R^r_{\Delta(2)}$,
\[
\ind_{\ZZ_r}(R^r_{\Delta(2)}) = H^{*\ge N-r+2}(B\ZZ_r).
\]
This follows from Theorem~\ref{thmIndexOfK}, which gives the index of the corresponding deleted join $R^{*r}_{\Delta(2)}$, and a reduction lemma due to Karasev \cite[Lemma 3.2]{Kar09} and independently Carsten Schultz (unpublished).
Then one can proceed exactly as in the proof of Volovikov's Theorem~\ref{thmVolovokovsTopTverberg}, see~\cite{Vol96}, which is based on Theorem 1 of his paper~\cite{Vol93}.

\subsection{Deleted joins versus deleted products} \label{secDelJoinsVsDelProducts}

An interesting question is which test map scheme is more powerful, the one coming from the deleted-product construction or the one from the deleted-join construction?

Let us first define general deleted products and deleted joins for a \emph{simplicial complex} $X$.
We define the \emph{$r$-fold $\ell$-wise deleted product} $X^r_{\Delta(\ell)}$ of the simplicial complex $X$ to be the subcomplex of
the cell complex $X^r$ containing only cells $F_1\times\ldots\times F_r$ such that the $F_i$ are $\ell$-wise disjoint, that is, no $\ell$ of them
have a point in common.
Analogously, we define the \emph{$r$-fold $\ell$-wise deleted join} $X^{*r}_{\Delta(\ell)}$ of the simplicial complex $X$ as the subcomplex of
the simplicial complex $X^{*r}$ that contains only those faces $F_1*\ldots*F_r$ such that the $F_i$ are $\ell$-wise disjoint.
Compare with \cite[Definition 6.3.1]{Mat07}.

Now we introduce general deleted products and deleted joins for a \emph{topological space} $Y$.
The \emph{$r$-fold $\ell$-wise deleted product} of the space $Y$ is
\[
Y^r_{\Delta(\ell)} := \{(y_1,\ldots,y_r)\in Y^r\st \textnormal{no $\ell$ of the $y_i$ are equal}\},
\]
while the \emph{$r$-fold $\ell$-wise deleted join} of $Y$ we define as
\[
Y^{*r}_{\Delta(\ell)} := \{\lambda_1y_1+\ldots+\lambda_r y_r \in Y^{*r}\st \textnormal{if $\lambda_1=\ldots=\lambda_r=\tfrac{1}{r}$ then no $\ell$ of the $y_i$ are equal}\}.
\]

In many applications we investigate the existence of a $\Sigma_r$-equivariant test map of deleted products
\begin{equation}
\label{eqGeneralTestMapDeletedProduct}
f^\times: X^r_{\Delta(\ell)} \toG{\Sigma_r} Y^r_{\Delta(k)},
\end{equation}
where $X$ is a simplicial complex and $Y$ is a space.
Here $\Sigma_r$ stands for the group of permutations on $r$ letters.
The corresponding $\Sigma_r$-equivariant test map for deleted joins would be
\begin{equation}
\label{eqGeneralTestMapDeletedJoin}
f^*: X^{*r}_{\Delta(\ell)} \toG{\Sigma_r} Y^{*r}_{\Delta(k)}.
\end{equation}

\noindent In the case when $Y=\RR^d$, for some $d$, the existence of the $\Sigma_r$-equivariant map
$f^\times$ implies the existence of the $\Sigma_r$-equivariant map $f^*$.
Indeed, if $f^\times: X^r_{\Delta(\ell)} \toG{\Sigma_r} (\RR^d)^r_{\Delta(k)}$ is given, then we can define
\[
f^*(\lambda_1 x_1+\ldots+\lambda_r x_r):=\sum_{i=1}^r\lambda_i y_i,
\]
where
\[
y_i := \left(\prod_{j=1}^rr\lambda_j\right)\cdot f_i^\times(x_1,\ldots,x_r)\in \RR^d.
\]
The constant factor $r^r$ is included in the definition of $y_i$ because $X^r_{\Delta(\ell)}$ can be seen as a subspace of $X^{*r}_{\Delta(\ell)}$ where all join coefficients are $\tfrac{1}{r}$. The cell $F_1\times\ldots\times F_r$ of the deleted product complex can be identified with the subspace
$\{\tfrac{1}{r}\cdot x_1+\ldots +\tfrac{1}{r}\cdot x_r~|~x_i\in F_i\}$
of the simplex $F_1*\ldots * F_r$ of the deleted join complex.

\noindent Therefore, in the case when $Y=\RR^d$, the deleted-product scheme \eqref{eqGeneralTestMapDeletedProduct} is at least as strong as the deleted-join scheme \eqref{eqGeneralTestMapDeletedJoin}.

Nevertheless, \emph{proving} the non-existence of $f^*$ might be easier than proving the non-existence of~$f^\times$.
For instance, if one only wants to argue with the high connectivity of the domain, then this is usually easier for $f^*$, see e.g.~\cite[Sections 5.5--5.8]{Mat07}.

Also the monotonicity of the Fadell--Husseini index sometimes puts a stronger condition on $f^*$ than on $f^\times$.
In particular, this affects Theorem~\ref{thmMain}. The range of $f^\times$ is $M^r_{\Delta(r)}$.
If $M=\RR^d$ then the corresponding index is
\[
\ind_{\ZZ_r}((\RR^d)^r_{\Delta(r)}) = H^{*\ge d(r-1)}(B\ZZ_r),
\]
since $(\RR^d)^r_{\Delta(r)}$ deformation retracts equivariantly to a fixed-point free sphere whose dimension is \mbox{$d(r-1)-1$}.
Hence we can show the non-existence of $f^\times$ using the monotonicity of the index.
However, for $M=S^d$ the index is smaller with respect to inclusion by the following proposition, so
the monotonicity of the index alone is not enough to prove Theorem~\ref{thmMain} in the deleted-product scheme.
\begin{proposition}
If $d\geq 2$, then
\[
\ind_{\ZZ_r}((S^d)^r_{\Delta(r)}) \subseteq H^{*\ge d(r-1)+1}(B\ZZ_r).
\]
\end{proposition}
\begin{proof}
We have to show that in the Leray--Serre spectral sequence associated to $EG\times_G(S^d)^r_{\Delta(r)}\to BG$, no non-zero differential hits the bottom row in filtration degree smaller or equal to $d(r-1)$. The $E_2$ entries are $H^*(BG,H^*((S^d)^r\wo\Delta))$, where $\Delta$ is the thin diagonal in $(S^d)^r$. Now,
\[
H^i((S^d)^r\wo\Delta)\iso H_{dr-i}((S^d)^r,\Delta)\iso H^{dr-i}((S^d)^r,\Delta).
\]
From the long exact sequence in cohomology of the pair $((S^d)^r,\Delta)$,
\[
\ldots\to H^*((S^d)^r,\Delta)\to H^*((S^d)^r) \to H^*(\Delta)\to\ldots,
\]
we see that $H^{dr}((S^d)^r,\Delta)=\FF_r$, $H^d((S^d)^r,\Delta)\iso \FF_r[\ZZ_r]/(1+t+\ldots+t^r)\FF_r$, and for $d<j<dr$ we have $H^j((S^d)^r,\Delta)\iso\FF_r[\ZZ_r]^{\oplus\alpha_j}$, where $\alpha_j\geq 0$ depends on $j$.
Now, $H^i(BG;\FF_p[\ZZ_p])$ is zero for $i\geq 1$.
Therefore the first non-zero row (up to the $0$-column entries) in the spectral sequence above the bottom row is the $d(r-1)$-row.
Thus the first element in the bottom row that is hit by a differential has degree at least $d(r-1)+1$.
\end{proof}

On the other hand, the monotonicity of the Fadell--Husseini index proves the non-existence of $f^*$ for $M=S^d$, since $(S^d)^{*r}_{\Delta(r)}$ deformation retracts equivariantly to an $(N-1)$-dimensional fixed-point free sphere, whose index is equal to $H^{*\ge N}(B\ZZ_r)$.

\noindent So, in this context, the deleted-join scheme is stronger than the deleted-product scheme.

\smallskip

\noindent\textbf{Acknowledgements.}
We are grateful to the referee for critical comments that improved differen aspects of the paper, and to Aleksandra, Julia, and Torsten for constant
support.


\end{document}